\documentclass[10pt]{article}

\usepackage{latexsym}
\usepackage{amssymb}
\usepackage{amsthm}
\usepackage{amscd}
\usepackage{amsmath}
\usepackage{mathrsfs}
\usepackage{graphicx}
\usepackage{hyperref}
\usepackage{shuffle}
\usepackage{mathtools}
\usepackage[bbgreekl]{mathbbol}
\usepackage{mathdots}
\usepackage{ytableau}
\usepackage{enumerate}
\usepackage[all]{xy}
\usepackage{tikz-cd}
\usepackage[colorinlistoftodos]{todonotes}
\usetikzlibrary{chains,scopes}
\usetikzlibrary{shapes.geometric,positioning}

\DeclareSymbolFontAlphabet{\mathbb}{AMSb}
\DeclareSymbolFontAlphabet{\mathbbl}{bbold}

\newcommand{\SetSSYT}{\ensuremath\mathrm{SetSSYT}}

\newcommand{\SetSSMT}{\ensuremath\mathrm{SetSSMT}}

\numberwithin{equation}{section}
\allowdisplaybreaks[1]

\theoremstyle{definition}
\newtheorem* {theorem*}{Theorem}
\newtheorem* {conjecture*}{Conjecture}
\newtheorem{theorem}{Theorem}[section]
\newtheorem{thmdef}[theorem]{Theorem-Definition}

\theoremstyle{definition}

\newtheorem* {remark*}{Remark}
\newtheorem* {example*}{Example}

\newtheorem{lemma}[theorem]{Lemma}
\theoremstyle{definition}
\newtheorem{definition}[theorem]{Definition}
\theoremstyle{definition}

\newtheorem{proposition}[theorem]{Proposition}
\newtheorem{corollary}[theorem]{Corollary}

\newtheorem{remark}[theorem]{Remark}
\theoremstyle{definition}
\newtheorem {example}[theorem]{Example}
\theoremstyle{definition}

\theoremstyle{definition}

\theoremstyle{definition}

\def\({\left(}
\def\){\right)}
\newcommand{\sP}{\mathscr{P}}
\newcommand{\CC}{\mathbb{C}}

\def\NN{\mathbb{N}}

\def\CC{\mathbb{C}}

\def\ZZ{\mathbb{Z}}

\def\spanning{\textnormal{-span}}

\newcommand{\cL}{\mathcal{L}}

\def\fk{\mathfrak}

\def\barr{\begin{array}}
\def\earr{\end{array}}
\def\ba{\begin{aligned}}
\def\ea{\end{aligned}}
\def\be{\begin{equation}}
\def\ee{\end{equation}}

\def\qquand{\qquad\text{and}\qquad}
\def\quand{\quad\text{and}\quad}

\def\ds{\displaystyle}

\def\PP{\mathbb{P}}
\def\MM{\mathbb{M}}

\def\fkS{\fk S}
\def\fkG{\fk G}

\def\ben{\begin{enumerate}}
\def\een{\end{enumerate}}

\def\fpf{{\textsf {FPF}}}

\newcommand{\wfpf}{\Theta}

\def\csp{c^\Sp }

\def\Asc{\mathrm{Asc}}

\def\SS{\mathrm{SD}}

\newcommand{\Fl}{\textsf{Fl}}

\newcommand{\Sp}{\mathsf{Sp}}

\def\GP{G\hspace{-0.2mm}P}

\def\YY{\mathrm{D}}

\def\lessF{\lessdot_F}

\newcommand{\myarc}[2]{ \ar @/^#1pc/ @{-} [#2] }
\def\arcstop{\endxy\ }
\def\arcstart{\ \xy<0cm,.2cm>\xymatrix@!@R=.1cm@C=.2cm }
\newcommand{\arcdiagram}[1]{\boxed{\ \arcstart{ \\ #1}\arcstop\ }}

\def\ellfpf{\ell_\fpf}

\def\t{\fk t}
\def\u{\fk u}

\def\t{\mathbf{t}}

\def\fku{\fk u}

\def\Sfpf{{\fk S}^{\Sp}}
\def\Gfpf{\mathfrak{G}^\Sp}

\def\GSp{\GP^{\Sp}}

\def\dearc{\textsf{dearc}}

\def\*{\mathbin{\hat\circ}}

\def\SLambda{\sP_{\textsf{strict}}}

\def\varpi{\partial^{(\beta)}}
\def\vartheta{\pi^{(\beta)}}

\def\shSp{\lambda^{\Sp}}

\def\DSp{D^{\Sp}}

\def\ISp{I^{\textsf{FPF}}}

\newcommand{\arc}[2]{ \ar @/^#1pc/ @{-} [#2] }
\def\arcstop{\endxy\ }
\def\arcstart{\ \xy<0cm,-.06cm>\xymatrix@R=.05cm@C=.35cm }
\newcommand{\arcstartc}[1]{\ \xy<0cm,-.15cm>\xymatrix@R=.1cm@C=#1cm}

\begin{document}
\title{On some properties of symplectic Grothendieck polynomials}
\author{
Eric Marberg\thanks{The first author was supported by Hong Kong RGC Grant ECS 26305218.} \\ HKUST \\ {\tt eric.marberg@gmail.com}
\and
Brendan Pawlowski \\ University of Southern California \\ {\tt br.pawlowski@gmail.com}
}

\date{}

\maketitle

\begin{abstract}
Grothendieck polynomials, introduced by Lascoux and Sch\"utzenberger,
are certain $K$-theory representatives for Schubert varieties. Symplectic Grothendieck polynomials, described more recently by Wyser and Yong,  represent the $K$-theory classes of orbit closures for the complex symplectic group acting on the complete flag variety. We prove a transition formula for symplectic Grothendieck polynomials and study their stable limits. We show that each of the $K$-theoretic Schur $P$-functions of Ikeda and Naruse arises from a limiting procedure applied to symplectic Grothendieck polynomials representing certain ``Grassmannian'' orbit closures.
\end{abstract}


\setcounter{tocdepth}{2}
\tableofcontents

\section{Introduction}



Let $n$ be a positive integer.
The $K$-theory ring of the variety $\Fl_n$ of complete flags in $\CC^n$ is isomorphic to a quotient of a polynomial ring \cite[\S2.3]{KM}.
Under this correspondence, the \emph{Grothendieck polynomials} $\fkG_w$ represent the classes of the structure sheaves of Schubert varieties. The results in this paper concern a family of \emph{symplectic Grothendieck polynomials} $\Gfpf_z$ which similarly represent the $K$-theory classes of the orbit closures of the complex symplectic group acting on $\Fl_n$.

The Grothendieck polynomials $\fkG_w$ lie in $\ZZ[\beta][x_1, x_2, \ldots]$, where $\beta, x_1, x_2, \ldots$ are commuting indeterminates,
and are indexed by elements $w$ of the group $S_{\infty}$ of permutations of the positive integers $\PP:=\{1,2,3,\dots\}$ with finite support. Lascoux and Sch\"utzenberger
first defined these polynomials in a slightly different form in \cite{Lascoux1990,LS1983}. Setting $\beta = 0$ 
transforms Grothendieck polynomials to \emph{Schubert polynomials}, which represent the Chow classes of Schubert varieties.

 Lenart \cite{Lenart}, extending work of Lascoux \cite{Lascoux2000}, proved a ``transition formula'' expressing any product $x_k \fkG_w$ as a finite linear combination of Grothendieck polynomials; the Bruhat order on $S_{\infty}$ controls which terms appear.
A nice corollary of Lenart's result is that the set of Grothendieck polynomials form a $\ZZ[\beta]$-basis
for the polynomial ring $\ZZ[\beta][x_1, x_2, \ldots]$ (see Corollary~\ref{basis-cor}).

For $w \in S_\infty$ and $m \geq 0$, let $1^m \times w \in S_{\infty}$ denote the permutation sending $i \mapsto i$ for $i \leq m$ and $m+i \mapsto m+w(i)$ for $i > m$. The \emph{stable Grothendieck polynomial} of $w \in S_\infty$ is then given by
\begin{equation} \label{eq:intro-Gw}
    G_w := \lim_{m \to \infty} \fkG_{1^m \times w} \in \ZZ[\beta][[x_1, x_2, \ldots]].
\end{equation}
Results of Fomin and Kirillov \cite{FK1994} show that this limit converges in the sense of formal power series to a well-defined symmetric function. 
Despite its name, $G_w$ is a power series rather than a polynomial.

Of particular interest are the stable Grothendieck polynomials $G_{\lambda} := G_{w_\lambda}$ where $w_{\lambda}$ is the \emph{Grassmannian} permutation associated to an integer partition $\lambda$   (see \S \ref{subsec:G-lambda}). The $G_{\lambda}$'s represent structure sheaves of Schubert varieties in a Grassmannian \cite{Buch2002} and are natural ``$K$-theoretic'' generalizations of Schur functions. One can deduce from the transition formula for $\fkG_w$ that $G_w$ is an $\NN[\beta]$-linear combination of $G_{\lambda}$'s, and the \emph{Hecke insertion algorithm} of \cite{BKSTY} leads to a combinatorial description of the coefficients in this expansion.

 The \emph{symplectic Grothendieck polynomials} $\Gfpf_z$ are a second family of polynomials in $\ZZ[\beta][x_1, x_2, \ldots]$, which now represent the $K$-theory classes of the orbit closures of the complex symplectic group acting on $\Fl_n$ for even $n$. They are indexed by elements $z$ of the set $\ISp_{\infty}$ of bijections $z : \PP \to \PP$ such that $z = z^{-1}$, $z(i) \neq i$ for all $i \in \PP$, and $z(2i-1) = 2i$ for all sufficiently large $i$.  
(We think of $\ISp_\infty$ as the set of fixed-point-free involutions of the positive integers with ``finite support.'')
 Wyser and Yong first considered these polynomials in \cite{WyserYong}, but their definition differs from ours by a minor change of variables. Setting $\beta = 0$ gives the \emph{fixed-point-free involution Schubert polynomials} studied in \cite{HMP1, HMP5, WyserYong}.

Our first main result, Theorem~\ref{sp-lenart-thm}, is an analogue of Lenart's transition formula for symplectic Grothendieck polynomials. This is somewhat more complicated than Lenart's identity, involving multiplication of $\Gfpf_z$ by two indeterminates $x_k$ and $x_{z(k)}$;
the corresponding proof is also more involved.
Nevertheless, there is a surprising formal similarity between the two transition equations.
 A variant of Bruhat order again plays a key role.

This paper is a sequel to \cite{MP}, where we showed that the
 natural analogue of the stable limit \eqref{eq:intro-Gw} for symplectic Grothendieck polynomials
defines a symmetric formal power series $\GSp_{z}$ for each $z \in \ISp_{\infty}$. 
Results of the first author \cite{Mar} show that $\GSp_z$ is a finite $\NN[\beta]$-linear combination of 
 Ikeda and Naruse's \emph{K-theoretic Schur $P$-functions} $\GP_{\lambda}$ \cite{IkedaNaruse}. 
 Here we prove an important related fact: each $\GP_{\lambda}$ occurs as $\GSp_{z_{\lambda}}$ where $z_{\lambda} \in \ISp_\infty$ is the \emph{FPF-Grassmannian} involution corresponding to $\lambda$. See Theorem~\ref{f-grass-thm} for the precise statement.

Every symmetric power series in $\ZZ[\beta][[x_1, x_2, \ldots]]$ can be written as a possibly infinite $\ZZ[\beta]$-linear combination of stable Grothendieck polynomials.
One application of the preceding paragraph
is a proof that 
each $K$-theoretic Schur $P$-function $\GP_\lambda$ is a finite sum of $G_\mu$'s with
coefficients in $\NN[\beta]$.
It is also possible to deduce this fact from the results in \cite{HKPWZZ,PylPat},
though the derivation is slightly roundabout;
see the remark after Corollary~\ref{cor:GP-into-G}.

A brief outline of the rest of this article is as follows.
Section~\ref{divdiff-sect} covers some background material on permutations, divided difference operators, 
and Grothendieck polynomials.
In Section~\ref{tr-sect} we review Lenart's transition formula for $\fkG_w$ and then prove its symplectic analogue.
Section~\ref{stable-sect}, finally, contains our results on symplectic stable Grothendieck polynomials.

\section{Preliminaries}\label{divdiff-sect}

This section includes a few preliminaries and sets up most of our notation.
We write $\NN = \{0,1,2,\dots\}$ and $\PP = \{1,2,3,\dots\}$ for the sets of nonnegative and positive integers,
and define $[n] := \{1,2,\dots,n\}$ for $n \in \NN$.
Throughout, 
the symbols $\beta$, $x_1$, $x_2$, $\dots$ denote commuting indeterminates.

\subsection{Permutations}

For $i \in \PP$, define $s_i = (i,i+1)$ to be the permutation of $\PP$ interchanging $i$ and $i+1$.
These simple transpositions generate the infinite Coxeter group 
$S_\infty := \langle s_i : i \in \PP\rangle$ of permutations of $\PP$ with finite support,
as well as the finite subgroups $S_n := \langle s_1,s_2,\dots,s_{n-1}\rangle$ for each $n \in \PP$.

The \emph{length} of $w \in S_\infty$
is $\ell(w) := | \{ (i,j) \in \PP\times \PP : i<j\text{ and }w(i)>w(j)\}|.$
This finite quantity is also the minimum number of factors in any expression for $w$ as a product of simple transpositions.

We represent elements of $S_\infty$ in one-line notation
by identifying a word $w_1w_2\cdots w_n$ that has $\{w_1,w_2,\dots,w_n\} =[n]$
with the permutation $w \in S_\infty$ that has $w(i) = w_i$ for $i \in [n]$ and $w(i) =i$ for all integers $i>n$.

\subsection{Divided difference operators}\label{dd-sect}

Let $\cL = \ZZ[\beta][x_1^{\pm1},x_2^{\pm1},\dots]$ denote the ring of Laurent polynomials
in the variables $x_1$, $x_2$, $\dots$ with coefficients in $\ZZ[\beta]$.
Given $i \in \PP$ and $f \in \cL$,
write $s_i f$ for the Laurent polynomial formed from $f$ by interchanging the variables $x_i$ and $x_{i+1}$.
This operation extends to a group action of $S_\infty$ on $\cL$.
For $i \in \PP$, 
the \emph{divided difference operators} $\partial_i $ and $\varpi_i$ are the maps $\cL \to \cL$ given by
\be \partial_i f = \tfrac{f -s_i f}{x_i-  x_{i+1}}\quand \varpi_i f = \partial_i( (1+\beta x_{i+1}) f) 
= -\beta f + (1+\beta x_i) \partial_i f.
\ee
Both operators preserve the subring of polynomials $\ZZ[\beta][x_1,x_2,\dots]\subset \cL$.

Some identities are useful for working with these maps.
All formulas involving $\varpi_i$
reduce to formulas involving $\partial_i$ on setting $\beta=0$.
Fix $i \in \PP$ 
and $f,g \in \cL$.
Then \be\label{leib-eq}
\varpi_i(fg) = s_i f \cdot (\varpi_i g + \beta g) + \varpi_i f \cdot g
\ee
and we have $\partial_i f=0$ and $\varpi_i f = -\beta f$ if and only if $s_if=f$,
in which case 
\be\label{iwc-eq}
\partial_i(fg) = f \cdot \partial_i g
\quand
\varpi_i(fg) = f \cdot \varpi_i g.\ee
Moreover, one has $\partial_i \partial_i = 0$ and $\varpi_i \varpi_i = -\beta\varpi_i$. 
Both families of operators
satisfy the usual braid relations for $S_\infty$, meaning that we have
\be
\varpi_i \varpi_j = \varpi_j \varpi_i
\quand \varpi_i \varpi_{i+1}\varpi_i = \varpi_{i+1}\varpi_i \varpi_{i+1}
\ee
for all $i,j \in \PP$ with $|i-j|>1$.
If $w \in S_\infty$ then we can therefore define
\[\partial_w := \partial_{i_1}\partial_{i_2}\cdots \partial_{i_l}\qquand \varpi_w := \varpi_{i_1}\varpi_{i_2}\cdots \varpi_{i_l}\]
where $w = s_{i_1}s_{i_2} \cdots s_{i_l}$ is any \emph{reduced expression},
i.e., a minimal length factorization of $w$ as a product of simple transpositions. 

 \subsection{Grothendieck polynomials}
 
The following definition of \emph{Grothendieck polynomials} originates in \cite{FK1994}.

\begin{thmdef}[Fomin and Kirillov \cite{FK1994}]
\label{groth-thmdef}
There exists a unique family $\{ \fkG_w \}_{w \in S_\infty}\subset \ZZ[\beta][x_1,x_2,\dots]$ with 
$\fkG_{n\cdots 321} = x_1^{n-1} x_2^{n-2} \cdots x_{n-1}^1$
for all $n \in \PP$
and such that 
$\varpi_i \fkG_w = \fkG_{w s_i}$ for $i \in \PP$ with $w(i) > w(i+1)$.
\end{thmdef}

 Note that it follows that 
$\varpi_i \fkG_w =-\beta \fkG_w$ if $w(i) < w(i+1)$.

\begin{example}
The Grothendieck polynomials for $w \in S_3$ are 
\[
\ba
\fkG_{123} &= 1, \\
\fkG_{213} &= x_1, \\
\ea
\qquad\quad
\ba
\fkG_{132} & =x_1 + x_2 + \beta x_1x_2, \\
\fkG_{231} &= x_1x_2, 
\ea
\qquad\quad
\ba
\fkG_{312} &= x_1^2, \\
\fkG_{321} &= x_1^2 x_2.
\ea
\]
\end{example}

We typically suppress the parameter $\beta$ in our notation,
but for the moment write $\fkG_w^{(\beta)} = \fkG_w$ for $w \in S_\infty$.
The \emph{Schubert polynomial} $\fkS_w$ of a permutation $w \in S_\infty$ (see \cite[Chapter 2]{Manivel})
is then $\fkG_w^{(0)}$. The polynomials $\{\fkS_w\}_{w \in S_\infty}$ are a $\ZZ[\beta]$-basis for 
$\ZZ[\beta][x_1,x_2,\dots]$ \cite[Proposition 2.5.4]{Manivel} so the Grothendieck polynomials are linearly independent.

Some references use the term ``Grothendieck polynomial'' to refer to the polynomials $\fkG_w^{(-1)}$.
One loses no generality in setting $\beta=-1$ since one can show by downward induction on permutation length that
\be\label{beta-rel} (-\beta)^{\ell(w)}\fkG_w^{(\beta)} =  \fkG_w^{(-1)}(-\beta x_1, -\beta x_2,\dots).\ee
Thus, it is straightforward to translate formulas in $\fkG_w^{(-1)}$ to formulas 
in $\fkG_w^{(\beta)}$.

\subsection{Symplectic Grothendieck polynomials}

Let $\Theta : \PP \to \PP$ be the map 
sending $i \mapsto i - (-1)^i$,
so that $\Theta =(1,2)(3,4)(5,6)\cdots.$
Define 
$\ISp_\infty := \{ w^{-1}\Theta w : w \in S_\infty\}.$
The elements of $\ISp_\infty$ are the involutions of the positive integers
that have no fixed points and that agree with $\Theta$ at all sufficiently large values of $i$.
We represent elements of $\ISp_\infty$ in one-line notation by identifying
a word $z_1z_2\cdots z_n$, satisfying $\{z_1,z_2,\dots,z_n\} = [n]$ and
$z_i = j$ if and only if $z_j = i \neq j$,
with the involution $z \in \ISp_\infty$ that has $z(i) = z_i$ for $i\in [n]$ and $z(i) = \Theta(i)$ for $i > n$. 

The symplectic analogues of $\fkG_w$ introduced below
were first studied by Wyser and Yong in a slightly different form; see \cite[Theorems  3 and 4]{WyserYong}.
The characterization given here combines \cite[Theorem 3.10 and Proposition 3.11]{MP}.

\begin{thmdef}[\cite{MP,WyserYong}]
\label{sp-thm1}
There exists a unique family $\{ \Gfpf_z \}_{z \in \ISp_\infty}\subset\ZZ[\beta][x_1,x_2,\dots]$
with
$\Gfpf_{n\cdots 321} = \prod_{1 \leq i < j \leq n - i } (x_i+ x_j+\beta x_ix_j)$
for all $n \in 2\PP$ 
and such that 
$\varpi _i \Gfpf_z =  \Gfpf_{s_izs_i}$ for $i \in \PP$ with
$i+1\neq z(i) > z(i+1) \neq i $.
\end{thmdef}

The elements of this family are the \emph{symplectic Grothendieck polynomials}
described in the introduction.
If $i \in \PP$ is such that $z(i) < z(i+1)$ or $i+1=z(i) > z(i+1)=i$
then $\varpi _i \Gfpf_z  = -\beta \Gfpf_z$ \cite[Proposition 3.11]{MP}.


\begin{example}\label{sp-ex1}
The polynomials for $z \in \ISp_4 := \{ w\Theta w^{-1} : w \in S_4\}$ are 
\[
\ba
\Gfpf_{2143} & = 1,\\
\Gfpf_{3412} &= x_1 + x_2 + \beta x_1x_2, \\
\Gfpf_{4321} &= x_1^2 + x_1 x_2 + x_1 x_3 + x_2 x_3 +  2\beta x_1 x_2 x_3 + \beta x_1^2 x_2 + \beta x_1^2 x_3 + \beta^2 x_1^2 x_2 x_3.
\ea
\]
The smallest example of $\Gfpf_z$ where $z$ is not $\Sp$-dominant
(see Theorem~\ref{dom-thm})
 is
\[
{\small\ba
\Gfpf_{351624} =\ 
& x_1^2 + 2 x_1 x_2 + x_2^2 + x_1 x_3 + x_2 x_3 + x_1 x_4 + x_2 x_4 + 2 \beta x_1^2 x_2
\\ & + 2 \beta x_1 x_2^2 + \beta x_1^2 x_3 + 3 \beta x_1 x_2 x_3 + \beta x_2^2 x_3 + \beta x_1^2 x_4+ 3 \beta x_1 x_2 x_4 
\\&  + \beta x_2^2 x_4 + \beta x_1 x_3 x_4 + \beta x_2 x_3 x_4 + \beta^2 x_1^2 x_2^2 + 2 \beta^2 x_1^2 x_2 x_3 + 2 \beta^2 x_1 x_2^2 x_3
\\&  + 2 \beta^2 x_1^2 x_2 x_4 + 2 \beta^2 x_1 x_2^2 x_4 + \beta^2 x_1^2 x_3 x_4 + 3 \beta^2 x_1 x_2 x_3 x_4 + \beta^2 x_2^2 x_3 x_4 
\\& + \beta^3 x_1^2 x_2^2 x_3 + \beta^3 x_1^2 x_2^2 x_4 + 2 \beta^3 x_1^2 x_2 x_3 x_4 + 2 \beta^3 x_1 x_2^2 x_3 x_4 + \beta^4 x_1^2 x_2^2 x_3 x_4.
\ea}
\]

\end{example}

Setting $\beta=0$ transforms $\Gfpf_z$ to the \emph{fixed-point-free involution Schubert polynomials} $\Sfpf_z$ studied in \cite{HMP1,HMP3,HMP5,WyserYong}.
Since the family $\{\Sfpf_z\}_{z \in \ISp_\infty}$ is linearly independent, 
$\{\Gfpf_z\}_{z \in \ISp_\infty}$ is also linearly independent.

We  need one other preliminary result concerning the polynomials $\Gfpf_z$.
    The \emph{symplectic Rothe diagram} of an involution $z \in \ISp_\infty$ is the set of pairs
    \begin{equation*}
        \DSp(z) := \{(i, z(j)) : \text{$(i,j) \in [n] \times [n]$ and $z(i) > z(j) < i < j$}\}.
    \end{equation*}
An element $z \in \ISp_\infty$ is \emph{$\Sp$-dominant} if
$ \DSp(z) = \{   (i+j,j) \in \PP \times [k] : 1 \leq i \leq \mu_j\}$
for a strict partition $\mu = (\mu_1 > \mu_2 > \dots > \mu_k > 0)$.
This condition holds, for example, when $z=n\cdots 321$ for any $n\in 2\PP$.
\begin{theorem}[{\cite[Theorem 3.8]{MP}}]
 \label{dom-thm}
If $z \in \ISp_\infty$ is $\Sp$-dominant then 
\[\Gfpf_z = \prod_{(i,j) \in \DSp(z)} (x_i+ x_j+\beta x_ix_j).\]
\end{theorem}

\section{Transition equations}\label{tr-sect}

Lenart \cite{Lenart} derives a formula expanding the product $x_k \fkG_v$ for $k \in \PP$ and $v \in S_\infty$ in terms of other Grothendieck polynomials.
In this section, we prove a similar identity for symplectic Grothendieck polynomials.

\subsection{Lenart's transition formula}

We recall Lenart's formula to motivate our new results.
Given $v \in S_\infty$ and $k \in \PP$, define $P_k(v)$ to be the set of 
all permutations in $S_\infty$ of the form
\[w = v (a_1,k)(a_2,k)\cdots(a_p,k)(k,b_1)(k,b_2)\cdots (k,b_q)\]
where $p,q \in \NN$ and
$a_p < \dots < a_2<a_1 < k < b_q < \dots < b_2 < b_1,$
and the length increases by exactly one 
upon multiplication by each transposition. 
Differing slightly from the convention in \cite{Lenart}, we allow the case $p=q=0$ so $w \in P_k(v)$.
Given $w \in P_k(v)$ define $\epsilon_k(w,v) = (-1)^{p}$.
This notation is well-defined since 
 $p$ can be recovered from $w \in P_k(v)$ as the number of indices $i<k$ with $v(i) \neq w(i)$.

\begin{theorem}[{\cite[Theorem 3.1]{Lenart}}]
\label{lenart-thm}
If $v \in S_\infty$ and $k \in \PP$ then
\[(1+\beta x_k) \fkG_v = \sum_{w \in P_k(v)} \epsilon_k(w,v) \beta^{\ell(w)-\ell(v)} \fkG_w.\]
\end{theorem}

The cited theorem of Lenart applies to the case when $\beta=-1$,
but this is equivalent to the given identity for generic $\beta$ by \eqref{beta-rel}.

\begin{example}
Taking $v=13452 \in S_\infty$ and $k=3$ in Theorem~\ref{lenart-thm} gives
\[\ba 
(1 + \beta x_3) \fkG_{13452} &=
\fkG_{13452} +
 \beta\fkG_{13542} - \beta\fkG_{14352} -
\beta^2\fkG_{14532} + \beta^2\fkG_{34152} 
\\&\quad+
\beta^3\fkG_{34512} + \beta^3\fkG_{34251} +
\beta^4\fkG_{34521}.
\ea
\]
This reduces to \cite[Example 3.9]{Lenart} on setting $\beta=-1$.
\end{example}

Lenart's formula implies that $x_k \fkG_v$ is a finite $\ZZ[\beta]$-linear combination of $\fkG_w$'s.
By starting with $v=1$ so that $\fkG_v=1$, we deduce that any monomial 
in $\ZZ[\beta][x_1,x_2,\dots]$ is a finite linear combination of Grothendieck polynomials.
Since these functions are also linearly independent, the following holds:

\begin{corollary}\label{basis-cor}
The set $\{\fkG_w\}_{w\in S_\infty}$ is a $\ZZ[\beta]$-basis for $\ZZ[\beta][x_1,x_2,\dots]$.
\end{corollary}

\begin{remark}
This corollary is nontrivial since $\fkG_w$ is an inhomogeneous polynomial of the form 
$\fkS_w + (\text{ terms of degree greater than $\ell(w)$ })$.
Since $\{ \fkS_w\}_{w \in S_\infty}$ is a $\ZZ$-basis for $\ZZ[x_1,x_2,\dots]$,
it follows that any polynomial in $\ZZ[\beta][x_1,x_2,\dots]$ can be inductively expanded in terms of Grothendieck polynomials.
However, it is not clear \emph{a priori} that such an expansion will terminate in a finite sum.
\end{remark}

For $v,w \in S_\infty$, write $v \lessdot w$ if $\ell(w) = \ell(v)+1$ and $v^{-1}w  = (i,j)$ is a transposition for some positive integers $i<j$.
It is well-known that if $w \in S_\infty$ and $i,j \in \PP$ are such that $i<j$,
then $w \lessdot w(i,j)$ if and only if $w(i) < w(j)$ and no integer $e$ has $i<e<j$ and $w(i) < w(e) < w(j)$.

For distinct integers $i,j \in \PP$, let $\t_{ij}$ be the linear operator, acting on the right,
with $\fkG_w \t_{ij} = \fkG_{w(i,j)}$ for $w \in S_\infty$.
We can restate Theorem~\ref{lenart-thm} as the following identity:

\begin{theorem}
\label{lenart-thm2}
Fix $v \in S_\infty$ and $k \in \PP$.
Suppose 
\[1 \leq j_1  < j_2<\dots < j_p < k
 < l_q  < \dots < l_2 < l_1\]
are the integers such that $v \lessdot v(j,k)$ and $v\lessdot v(k,l)$.
Then
\[ (1 + \beta x_k) \Bigl[ \fkG_v\cdot  (1 + \beta \t_{j_1k}) \cdots (1 + \beta \t_{j_pk})\Bigr] = \fkG_v\cdot (1 + \beta \t_{kl_1}) \cdots (1+ \beta \t_{kl_q}).
\]
\end{theorem}

\begin{proof}
After setting $\beta=-1$, 
this is a slight generalization of \cite[Corollary 3.10]{Lenart} (which is the main result of \cite{Lascoux2000}),
and has nearly the same proof.
Let $J = \{ j_1,j_2,\dots,j_p\}$ and $L = \{l_1,l_2,\dots,l_q\}$.
For subsets $E = \{ e_1 <  e_2<\dots <e_m \} \subset J$
and $F = \{ f_n < \dots  < f_2< f_1\} \subset L$
define $t_{E,k},t_{k,F}\in S_\infty$ by
\[
t_{A,k} = (e_1,k)(e_2,k)\cdots (e_m,k) 
\quand
 t_{k,B} = (k,f_1)(k,f_2)\cdots (k,f_n).
\]
One has $\ell(vt_{E,k}) = \ell(v) +|E|$ and $\ell(vt_{k,F}) = \ell(v) + |F|$
for all choices of $E \subset J$ and $F \subset L$.
Hence, by Theorem~\ref{lenart-thm}, we must show that 
\be\label{must-eq}
\sum_{E\subset J} \sum_{w \in P_k(vt_{E,k})} \epsilon_k(w,vt_{E,k}) \beta^{\ell(w)-\ell(v)} \fkG_w
=
\sum_{F \subset L}  \beta^{|F|} \fkG_{vt_{k,F}}.
\ee
Each permutation $w$ indexing the sum on the left can be written as
\[
w = v (i_1,k)(i_2,k)\cdots (i_m,k) (i_{m+1},k) \cdots (i_n,k)(k,i_{n+1})\cdots (k,i_r)
\]
for some indices with $i_1 < i_ 2 < \dots< i_m > i_{m+1} >\dots > i_n$ and 
$i_{n+1} > \dots > i_r > k$
and $\{i_1,i_2,\dots,i_m\} \subset J$.
Here, the set indexing the outer sum
on the left side of \eqref{must-eq} is $E = \{ i_1,i_2,\dots , i_m\}$.
If $n>0$ then each such $w$ appears twice 
with opposite associated signs $\epsilon_k(w,vt_{E,k})$; the two appearances correspond to 
$E = \{i_1,\dots,i_m\}$ and $E=\{i_1,\dots,i_{m-1}\}$.
The permutations $w$ that arise with $n=0$, alternatively, are exactly
the elements $vt_{k,F}$ for $F \subset L$, so \eqref{must-eq} holds.
\end{proof}

\subsection{Fixed-point-free Bruhat order}

For each involution $z \in \ISp_\infty$, let
\be
\label{ellfpf-eq}
 \ellfpf(z) =
| \{ (i,j) \in \PP\times \PP :z(i) > z(j) < i < j\}|.
\ee
One can check that if $z \in \ISp_\infty$ and $i \in \PP$ then
\be\label{ellfpf-eq2}
\ellfpf(s_izs_i)
=
\begin{cases} 
\ellfpf(z) + 1&\text{if }z(i) < z(i+1) \\
\ellfpf(z) & \text{if }i+1 = z(i) > z(i+1)=i \\
\ellfpf(z)-1 &\text{if }i+1 \neq z(i) > z(i+1) \neq i.
\end{cases}
\ee
It follows by induction that 
$
\ellfpf(z) = \min \{ \ell(w) : w \in S_\infty\text{ and } w^{-1} \Theta w  =z\}.
$

For $y,z \in \ISp_\infty$, we write 
$y \lessF z$
if $\ellfpf(z) = \ellfpf(y)+1$ and $z= tyt$ for a transposition $t \in S_\infty$.
The transitive closure of this relation is the \emph{Bruhat order} on $\ISp_\infty$ from \cite[\S4.1]{HMP5}.
One can give 
a more explicit characterization of $\lessF$:
 
 \begin{proposition}[{\cite[Proposition 4.9]{HMP3}}]
 \label{lessF-prop}
Suppose $y \in \ISp_\infty$, $i,j \in \PP$, and $i<j$.
\ben
\item[(a)] If $y(i) < i$ then $y \lessF (i,j) y(i,j)$ if and only if these properties hold:
\begin{itemize}
\item Either $y(i) < i < j < y(j)$ or $y(i) < y(j) < i < j$.
\item No integer $e$ has $ i<e<j $ and $y(i) < y(e) < y(j)$.
\end{itemize}

\item[(b)] If $j < y(j)$ then $y \lessF (i,j) y(i,j)$ if and only if these properties hold:
\begin{itemize}
\item Either $y(i) < i < j < y(j)$ or $ i < j < y(i) < y(j)$.
\item No integer $e$ has $ i<e<j $ and $y(i) < y(e) < y(j)$.
\end{itemize}
\een
\end{proposition}
 
\begin{remark}
Let $y \in \ISp_\infty$ and $i<j$ and $t=(i,j) \in S_\infty$. 
The cases when $y \lessF tyt$ correspond to the following pictures,
in which the edges indicate the cycle structure of the relevant involutions 
restricted to $\{i,j,y(i),y(j)\}$:
\[
y=\arcdiagram{
*{}    \myarc{0.9}{rr}  & *{}  \myarc{0.9}{rr}  & *{i} & *{j} 
}
\lessF
\arcdiagram{
*{}    \myarc{0.9}{rrr}  & *{}  \myarc{0.5}{r}  & *{i} & *{j} 
}=tyt,\]
\[ 
y=\arcdiagram{
*{}    \myarc{0.5}{r}  & *{i}   & *{j} \myarc{0.5}{r} & *{} 
}
\lessF
\arcdiagram{
*{}    \myarc{0.9}{rr}  & *{i}  \myarc{0.9}{rr}  & *{j} & *{} 
} = tyt,\]
\[
y=\arcdiagram{
*{i}    \myarc{0.9}{rr}  & *{j}  \myarc{0.9}{rr}  & *{} & *{} 
}
\lessF
\arcdiagram{
*{i}    \myarc{0.9}{rrr}  & *{j}  \myarc{0.5}{r}  & *{} & *{} 
}=tyt.
\]

\end{remark}

\subsection{Symplectic transitions}\label{sp-tr-sect}

For distinct $i,j \in \PP$,
define $\u_{ij}$ to be the linear operator 
with $\Gfpf_z \u_{ij} = \Gfpf_{(i,j)z(i,j)}$ for $z \in \ISp_\infty$.
One cannot hope for a symplectic version of Theorem~\ref{lenart-thm}
since products of the form $(1+\beta x_k) \Gfpf_z$ may fail to be linear combinations
of symplectic Grothendieck polynomials.
There is an analogue of Theorem~\ref{lenart-thm2}, however:

\begin{theorem}\label{sp-lenart-thm}
Fix $v \in \ISp_\infty$ and $j,k \in \PP$ with $v(k) = j<k = v(j)$.
Suppose 
\be\label{iill-eq} 1 \leq i_1  < i_2<\dots < i_p <j< k
 < l_q  < \dots < l_2 < l_1\ee
are the integers such that $v \lessF (i,j)v(i,j)$ and $v\lessF (k,l)v(k,l)$.
Then
\be\label{*} 
(1+\beta x_j)(1+\beta x_k) \Bigl[ \Gfpf_v\cdot  (1+\beta\u_{i_1j})(1+\beta\u_{i_2j})\cdots (1+\beta\u_{i_pj})\Bigr]
\ee
is equal to
\be\label{**}
\Gfpf_v\cdot (1+\beta\u_{kl_1})(1+\beta\u_{kl_2})\cdots (1+\beta\u_{kl_q}).
\ee
\end{theorem}

This is a generalization of  \cite[Theorem 4.17]{HMP3},
which one recovers by subtracting $\Gfpf_v$ from \eqref{*} and \eqref{**}, dividing
by $\beta$,
and then setting $\beta=0$.
These results belong to a larger family of similar formulas 
related to Schubert calculus; see also \cite{BilleyTransitions,LamShimozono,MarZha}.
Before giving the proof, we present one example.

\begin{example}
If $v = (1,2)(3,5)(4,8)(6,7) \in \ISp_\infty$ and $(j,k) = (3,5)$, then 
we have $\{i_1<i_2<\dots<i_p\} = \{2\}$ and $\{l_1>l_2>\dots>l_q\} = \{6,8\}$
and Theorem~\ref{sp-lenart-thm} is equivalent, after a few manipulations, to the claim that
$ (x_3 + x_5 + \beta x_3 x_5)  \Gfpf_{(1,2)(3,5)(4,8)(6,7)} + (1+\beta x_3)(1+\beta x_5) \Gfpf_{(1,3)(2,5)(4,8)(6,7)}$
is equal to
$\Gfpf_{(1,2)(3,8)(4,5)(6,7)} + \Gfpf_{(1,2)(3,6)(4,8)(5,7)} + \beta \Gfpf_{(1,2)(3,8)(4,6)(5,7)}.
$
\end{example}

\def\fff{\tau}

\begin{proof}[Proof of Theorem~\ref{sp-lenart-thm}]
The proof is by downward induction on $\ellfpf(v)$.
As a base case, suppose $v = n\cdots 321 \in \ISp_\infty$ where $n \in 2\PP$, so that $j = n+1-k$.
Then $p=0$, $q=1$, $l_1 = n+1$, and the theorem reduces to the claim that 
\[(x_j + x_{n+1-j} + \beta x_j x_{n+1-j}) \Gfpf_{n\cdots 321} = \Gfpf_{w}\]
for $w:=(k,n+1)v(k,n+1)$.
This follows from 
Theorem~\ref{dom-thm}
since $w$ is $\Sp$-dominant with $\DSp(w) =\DSp(v) \sqcup\{(n+1-j,j)\}$.

Now let $v \in \ISp_\infty\setminus\{\wfpf\}$ and $j,k \in \PP$ be arbitrary with $v(k) = j < k = v(j)$.
It is helpful to introduce some relevant notation. Define 
\[
\ba
\Pi^-(v,j,k) &:= \Gfpf_v\cdot  (1+\beta\u_{i_1j})(1+\beta\u_{i_2j})\cdots (1+\beta\u_{i_pj})
\\
\Pi^+(v,j,k) &:= \Gfpf_v\cdot (1+\beta\u_{kl_1})(1+\beta\u_{kl_2})\cdots (1+\beta\u_{kl_q})
\ea
\]
and let $\Asc^-(v,j,k) =\{i_1,i_2,\dots,i_p\}$ and $\Asc^+(v,j,k) = \{l_1,l_2,\dots,l_q\}$
where the indices $i_1,i_2,\dots,i_p$ and $l_1,l_2,\dots,l_q$ are as in \eqref{iill-eq}.
For each nonempty subset $A=\{a_1<a_2<\dots<a_m\} \subset \Asc^-(v,j,k)$, define 
\[\fff^-_A(v,j,k) = \sigma v\sigma^{-1},\quad\text{where $\sigma  := (a_1, a_2,\dots,a_m, j) \in S_\infty$.}\]
For each  nonempty subset $B=\{b_m<\dots<b_2<b_1\} \subset \Asc^+(v,j,k)$, define 
\[\fff^+_B(v,j,k) = \sigma v\sigma^{-1},
\quad\text{where $\sigma  := (b_1, b_2,\dots,b_m, k) \in S_\infty$.}
\]
For empty sets, we define $\fff^\pm_\varnothing(v,j,k) = v$.
It then follows from Proposition~\ref{lessF-prop}
that $\ellfpf(\fff^\pm_S(v,j,k))  = \ellfpf(v) + |S|$ for all choices of $S$, and we have
 \be\label{sumsum-eq}
\Pi^\pm(v,j,k) = \sum_{S\subset \Asc^\pm(v,j,k)} \beta^{|S|} \Gfpf_{\fff^\pm_S(v,j,k)}.
 \ee
 If we represent elements of $\ISp_\infty$ as arc diagrams, i.e., as perfect matchings on the positive integers 
with an edge for each 2-cycle, 
  then the elements $\fff^\pm_S(v,j,k)$ can be understood as follows.
  The arc diagram of $\fff^-_S(v,j,k)$ is formed from $v$ by cyclically shifting up the endpoints $S \sqcup \{j\}$. For example, if the relevant part of the arc diagram of $v$ appears as
    \[
  \arcstart{
  \\
  \\
    *{\bullet}  \arc{1.8}{rrrrrrr}    & *{\bullet}  \arc{1.35}{rrrrr}  & *{\ }  \arc{0.9}{rrr} &  *{\bullet} \arc{0.45}{r}        & *{\circ}    & *{\ }   & *{\circ}    &
*{\circ}    & *{\circ}  \arc{0.45}{r}   & *{\bullet }   & *{\ \ } \arc{1.35}{rrrrr}        & *{\ \ } \arc{0.9}{rrr}   & *{\ \ }  \arc{0.45}{r}    & *{\ \ }   & *{\ \ }  & *{\ \ }   
 \\   &  &  &  & i_1 & i_2 & i_3  & i_4 & j & k & l_3 & l_2 & l_1
} 
\arcstop
\]
where the elements of $S\sqcup\{j\}$ are labeled by $\circ $ while the elements of $v(S) \sqcup\{k\}$ are labeled by $\bullet$, then the arc diagram of $\fff^-_S(v,j,k)$ is
  \[
  \arcstart{ 
  \\
  \\
    *{\bullet}  \arc{2.025}{rrrrrrrr}    & *{\bullet}  \arc{1.575}{rrrrrr}  & *{\ }  \arc{0.9}{rrr} &  *{\bullet} \arc{0.9}{rrr}        & *{\circ} \arc{1.35}{rrrrr}    & *{\ }   & *{\circ}    &
*{\circ}    & *{\circ}     & *{\bullet }   & *{\ \ } \arc{1.35}{rrrrr}        & *{\ \ } \arc{0.9}{rrr}   & *{\ \ }  \arc{0.45}{r}    & *{\ \ }   & *{\ \ }  & *{\ \ }   
 \\   &  &  &  & i_1 & i_2 & i_3  & i_4 & j & k & l_3 & l_2 & l_1
} 
\arcstop.
\]
Similarly, 
  the arc diagram of $\fff^+_S(v,j,k)$ is formed from $v$ by cyclically shifting down the endpoints $\{k\} \sqcup S$. For example, if the arc diagram of $v$ is
    \[
  \arcstart{
  \\
  \\
    *{\ \ }  \arc{1.8}{rrrrrrr}    & *{\ \ }  \arc{1.35}{rrrrr}  & *{\ \ }  \arc{0.9}{rrr} &  *{\ \ } \arc{0.45}{r}        & *{\ \ }    & *{\ }   & *{\ \ }    &
*{\ \ }    & *{\bullet }  \arc{0.45}{r}   & *{\circ }   & *{\circ} \arc{1.35}{rrrrr}        & *{\circ} \arc{0.9}{rrr}   & *{\ \ }  \arc{0.45}{r}    & *{\ \ }   & *{\bullet}  & *{\bullet}   
 \\   &  &  &  & i_1 & i_2 & i_3  & i_4 & j & k & l_3 & l_2 & l_1
} 
\arcstop
\]
where the elements of $S\sqcup\{k\}$ are labeled by $\circ $ while the elements of $v(S) \sqcup\{j\}$ are labeled by $\bullet$, then the arc diagram of $\fff^-_S(v,j,k)$ is
    \[
  \arcstart{
  \\
  \\
    *{\ \ }  \arc{1.8}{rrrrrrr}    & *{\ \ }  \arc{1.35}{rrrrr}  & *{\ \ }  \arc{0.9}{rrr} &  *{\ \ } \arc{0.45}{r}        & *{\ \ }    & *{\ }   & *{\ \ }    &
*{\ \ }    & *{\bullet }  \arc{0.9}{rrr}   & *{\circ }  \arc{1.575}{rrrrrr}    &  *{\circ}      \arc{1.125}{rrrr}     &   *{\circ}  & *{\ \ }  \arc{0.45}{r}    & *{\ \ }   & *{\bullet}  & *{\bullet}   
 \\   &  &  &  & i_1 & i_2 & i_3  & i_4 & j & k & l_3 & l_2 & l_1
} 
\arcstop.
\]

Suppose the theorem holds for a given $v \in \ISp_\infty\setminus \{\wfpf\}$
in the sense that 
$(1+\beta x_j)(1+\beta x_k) \Pi^-(v,j,k) = \Pi^+(v,j,k)$
for all choices of $v(k) = j<k=v(j)$.
Let $d \in \PP$ be
any positive integer
with $d+1 \neq v(d) > v(d+1) \neq d$
and set
\[ w := s_d vs_d \in \ISp_\infty.\]
Choose integers $j,k\in\PP$ with $v(k) =j < k = v(j)$;
note that we cannot have $j=d<d+1=k$.
In view of the first paragraph, it is enough to show that 
\[
(1+\beta x_{j'})(1+\beta x_{k'}) \Pi^-(w,j',k') = \Pi^+(w,j',k')
\]
where $j' = s_d(j)$ and $k' = s_d(k)$.
There are seven cases to examine:
\begin{itemize}

\item \emph{Case 1:} Assume that $d+1<j$.
We must show that \[(1+\beta x_j)(1+\beta x_k) \Pi^-(w,j,k) = \Pi^+(w,j,k).\]
It suffices by \eqref{iwc-eq} to prove that 
$\varpi_d \Pi^\pm(v,j,k) = \Pi^\pm(w, j, k).$
The $+$ form of this claim is straightforward from Proposition~\ref{lessF-prop} and \eqref{sumsum-eq};
in particular, it holds that $\Asc^+(w,j,k) = \Asc^+(v,j,k)$.
For the other form, there are four subcases to consider:
\begin{itemize}
\item[(1a)] Assume that $d,d+1 \notin \Asc^-(v,j,k)$.
Since $v(d) >v(d+1)$, it is again straightforward from Proposition~\ref{lessF-prop} and \eqref{sumsum-eq}
to show that $\Asc^-(w,j,k) = \Asc^-(v,j,k)$ and $\varpi_d \Pi^-(v,j,k) = \Pi^-(w, j, k)$.

\item[(1b)] Assume that $d\in  \Asc^-(v,j,k)$ and $d+1 \notin \Asc^-(v,j,k)$.
Then \[\Asc^-(w,j,k) = \Asc^-(v,j,k) \setminus\{d\} \sqcup\{d+1\}\]
and $d+1 \neq \fff^-_S(v,j,k)(d) > \fff^-_S(v,j,k)(d+1) \neq d$ for all subsets $S \subset \Asc^-(v,j,k)$,
so we again have $\varpi_d \Pi^-(v,j,k) = \Pi^-(w, j, k)$.

\item[(1c)] Assume that $d\notin  \Asc^-(v,j,k)$ and $d+1 \in \Asc^-(v,j,k)$.
This can only occur if $d < k < v(d)$, so we have
\[\Asc^-(w,j,k) = \Asc^-(v,j,k) \setminus\{d+1\} \sqcup\{d\}.\]
From here, we deduce that $\varpi_d \Pi^-(v,j,k) = \Pi^-(w, j, k)$ by 
an argument similar to the one in case (1b). 

\item[(1d)] Assume that $d,d+1 \in \Asc^-(v,j,k)$.
Three situations are possible for the relative order of $d$, $d+1$, $v(d)$, and $v(d+1)$.
First suppose $v(d+1) < d<d+1<v(d)$.
Then 
\be\label{again-eq}\Asc^-(w,j,k) = \Asc^-(v,j,k) \setminus\{d\}\ee
and every $i \in \Asc^-(v,j,k)$ with $i<d$ must have $v(d)<v(i)<k$.
It follows that if $S \subset \Asc^-(v,j,k)$ then
\be\label{again-eq2} \varpi_d \Gfpf_{\fff^-_S(v,j,k)} = \begin{cases}
\Gfpf_{\fff^-_{S\setminus\{d\}}(v,j,k)}  &\text{if }d,d+1 \in S \\[-8pt]\\
-\beta \cdot \Gfpf_{\fff^-_S(v,j,k)} &\text{if }d\notin S,\ d+1\in S\\[-8pt]\\
\Gfpf_{\fff^-_{S}(w,j,k)}  &\text{if }d,d+1 \notin S \\[-8pt]\\
\Gfpf_{\fff^-_{S\setminus\{d\}\sqcup \{d+1\}}(w,j,k)}  &\text{if }d\in S,\ d+1 \notin S.
\end{cases}
\ee
If we have $v(d+1) < v(d) < d$ or $j < v(d+1) < v(d)<k$
then \eqref{again-eq} and \eqref{again-eq2} both still hold and follow by similar reasoning.
Combining these identities with \eqref{sumsum-eq} gives $\varpi_d \Pi^-(v,j,k) = \Pi^-(w, j, k)$.
\end{itemize}
We conclude from this analysis that $\varpi_d \Pi^\pm(v,j,k) = \Pi^\pm(w, j, k).$

\item \emph{Case 2:} Assume that $d+1=j$, so that $k<v(j-1)$.
We must show that \[(1+\beta x_{j-1})(1+\beta x_k) \Pi^-(w,j-1,k) = \Pi^+(w,j-1,k).\]
It follows from \eqref{leib-eq} that 
$
\varpi_{j-1}\((1+\beta x_{j})(1+\beta x_k) \Pi^-(v,j,k)\)
$
is equal to
\[
(1+\beta x_{j-1})(1+\beta x_k) \varpi_{j-1} \Pi^-(v,j,k) - \beta \cdot \Pi^+(v,j,k)
\]
and it is easy to see that 
$\varpi_{j-1} \Pi^-(v,j,k) =  \Pi^-(w,j-1,k)$.
Thus, it suffices to show that 
\be\label{needed-eq}(\varpi_{j-1}+\beta)\Pi^+(v,j,k) =  \Pi^+(w,j-1,k).\ee
It follows from Proposition~\ref{lessF-prop} that
$ \Asc^+(w,j-1,k) = \{ l \in \Asc^+(v,j,k) : l < v(j-1) \} \sqcup \{ v(j)\}.$
We deduce that if $S \subset \Asc^+(v,j,k)$ then 
\[  \varpi_{j-1} \Gfpf_{\fff^+_S(v,j,k)} = \begin{cases}
\Gfpf_{\fff^+_{S}(w,j-1,k)}  &\text{if }S\subset  \Asc^+(w,j-1,k) \\[-8pt]\\
-\beta \cdot \Gfpf_{\fff^+_S(v,j,k)} &\text{otherwise}.
\end{cases}
\]
If $v(j) \in S \subset \Asc^+(w,j-1,k) $ then 
$\fff^+_{S}(w,j-1,k) = \fff^+_{S\setminus\{v(j)\}}(v,j,k).$
Combining these identities with \eqref{sumsum-eq} shows the needed claim \eqref{needed-eq}.

\item \emph{Case 3:} Assume that $d=j$,
so that either $v(j+1) < j<j+1<k$ or $j<j+1<v(j+1)<k$.
We must show that \[(1+\beta x_{j+1})(1+\beta x_k) \Pi^-(w,j+1,k) = \Pi^+(w,j+1,k).\]
It follows from \eqref{leib-eq} that 
$
\varpi_{j}\((1+\beta x_{j})(1+\beta x_k) \Pi^-(v,j,k)\)
$
is equal to
\[
(1+\beta x_{j+1})(1+\beta x_k) (\varpi_{j}+\beta) \Pi^-(v,j,k).
\]
It is easy to deduce that 
$\varpi_{j} \Pi^+(v,j,k) =  \Pi^+(w,j+1,k)$ from Proposition~\ref{lessF-prop},
so  it suffices to show that 
\be\label{nnn-eq} (\varpi_{j}+\beta) \Pi^-(v,j,k) = \Pi^-(w,j+1,k).\ee
First assume $v(j+1) < j<j+1<k$. 
Then every $i \in \Asc^-(v,j,k)$ with $i < v(j+1)$ 
must have $j+1 < v(i) < k$, and $\Asc^-(w,j+1,k)$ is equal to 
\[ \{ i \in \Asc^-(v,j,k) : v(j+1)<v(i)\text{ if }v(j+1) \leq i<j\} \sqcup \{ j\}.
\]
As in Case 2, we deduce that if $S \subset \Asc^-(v,j,k)$ then 
\be\label{same-eq3}  \varpi_{j} \Gfpf_{\fff^-_S(v,j,k)} = \begin{cases}
\Gfpf_{\fff^-_{S}(w,j+1,k)}  &\text{if }S\subset  \Asc^-(w,j+1,k) \\[-8pt]\\
-\beta \cdot \Gfpf_{\fff^-_S(v,j,k)} &\text{otherwise}.
\end{cases}
\ee
On the other hand, if $j \in S \subset \Asc^-(w,j+1,k) $ then 
\be\label{same-eq4}
\fff^-_{S}(w,j+1,k) = \fff^-_{S\setminus\{j\}}(v,j,k).
\ee
Combining these identities with \eqref{sumsum-eq} gives \eqref{nnn-eq} as desired.
Alternatively,  if we have $j<j+1<v(j+1)<k$,
then 
\[
\Asc^-(w,j+1,k) = \{ i \in \Asc^-(v,j,k) : v(j+1)<v(i)<k\} \sqcup \{ j\}
\]
and we deduce by similar reasoning that the identities \eqref{same-eq3} and \eqref{same-eq4} both still hold,
so \eqref{nnn-eq} again follows.

\item \emph{Case 4:} Assume that $j<d$ and $d+1<k$.
We must show that \[(1+\beta x_{j})(1+\beta x_k) \Pi^-(w,j,k) = \Pi^+(w,j-1,k).\]
It suffices by \eqref{iwc-eq} to prove that 
$\varpi_d \Pi^\pm(v,j,k) = \Pi^\pm(w, j, k).$
There are three subcases to consider:
\ben
\item[(4a)] If $j < v(d) < k$ or $j<v(d+1)<k$ or $v(d+1)  < j < k < v(d)$ then
the desired identities are straightforward from Proposition~\ref{lessF-prop}.

\item[(4b)] Assume that $v(d+1) < v(d) < j$.
In this case it is easy to see that $\varpi_d \Pi^+(v,j,k) = \Pi^+(w, j, k)$
and if $v(d+1) \notin \Asc^-(v,j,k)$,
then
we likewise deduce that $\varpi_d \Pi^-(v,j,k) = \Pi^-(w, j, k).$
Suppose instead that $v(d+1) \in \Asc^-(v,j,k)$. 
We then also have $v(d) \in \Asc^-(v,j,k)$,
but no $i \in \Asc^-(v,j,k)$ is such that $v(d+1)<i<v(d)$,
and  \[
\Asc^-(w,j,k) = \Asc^-(v,j,k)\setminus\{v(d+1)\}.
\]
It follows that if $S \subset \Asc^-(v,j,k)$ then
\[ \varpi_d \Gfpf_{\fff^-_S(v,j,k)} = \begin{cases}
\Gfpf_{\fff^-_{S\setminus\{v(d)\}}(v,j,k)}  &\text{if }v(d),v(d+1) \in S \\[-8pt]\\
-\beta \cdot \Gfpf_{\fff^-_S(v,j,k)} &\text{if }v(d)\notin S,\ v(d+1)\in S\\[-8pt]\\
\Gfpf_{\fff^-_{S}(w,j,k)}  &\text{if }v(d+1) \notin S.
\end{cases}
\]
Combining this with \eqref{sumsum-eq} gives $\varpi_d \Pi^-(v,j,k) = \Pi^-(w, j, k).$

\item[(4c)] Assume that $ k < v(d+1) < v(d)$.
This is the mirror image of (4b) and we get
 $\varpi_d \Pi^\pm(v,j,k) = \Pi^\pm(w, j, k)$ by symmetric arguments.
\een

\item \emph{Case 5:} Assume that $d+1=k$, 
so that either $j < k-1<k < v(k-1)$ or $j < v(k-1) < k-1<k$.
We must show that \[(1+\beta x_{j})(1+\beta x_{k-1}) \Pi^-(w,j,k-1) = \Pi^+(w,j,k-1).\]
It follows from \eqref{leib-eq} that 
$
\varpi_{k-1}\((1+\beta x_{j})(1+\beta x_k) \Pi^-(v,j,k)\)
$
is equal to
\[
(1+\beta x_{j})(1+\beta x_{k-1}) \varpi_{j} \Pi^-(v,j,k) - \beta \cdot \Pi^+(v,j,k)
\]
and it is easy to deduce that 
$\varpi_{k-1} \Pi^-(v,j,k) =  \Pi^-(w,j,k-1)$.
It therefore suffices to show that 
$ (\varpi_{k-1}+\beta) \Pi^+(v,j,k) = \Pi^+(w,j,k-1)$.
The required argument is the mirror image of Case 3; we omit the details.

\item \emph{Case 6:} Assume that $d=k$.
We must show that \[(1+\beta x_{j})(1+\beta x_{k+1}) \Pi^-(w,j,k+1) = \Pi^+(w,j,k+1).\]
It follows from \eqref{leib-eq} that 
$
\varpi_{k}\((1+\beta x_{j})(1+\beta x_k) \Pi^-(v,j,k)\)
$ is equal to
\[
(1+\beta x_{j})(1+\beta x_{k+1}) (\varpi_{k}+\beta) \Pi^-(v,j,k)
\]
and it is easy to see that 
$\varpi_{k} \Pi^+(v,j,k) =  \Pi^+(w,j,k+1)$.
Thus, it suffices to show that 
$(\varpi_{k}+\beta) \Pi^-(v,j,k) =  \Pi^-(w,j,k+1).$
The required argument is the mirror image of Case 2; we omit the details.

\item \emph{Case 7:} Finally, assume that $k<d$.
We must show that \[(1+\beta x_{j})(1+\beta x_{k}) \Pi^-(w,j,k) = \Pi^+(w,j,k).\]
It suffices by \eqref{iwc-eq} to prove that 
$\varpi_d \Pi^\pm(v,j,k) = \Pi^\pm(w, j, k).$
The required argument is the mirror image of Case 1; we omit the details.

\end{itemize}
This case analysis completes our inductive proof. 
\end{proof}

\begin{corollary}
Suppose $v \in \ISp_\infty$ and $j,k \in \PP$ have $ j<k=v(j)$. Then 
\[ (1+\beta x_j)(1+\beta x_k)\Gfpf_v \in \ZZ[\beta]\spanning\left\{ \Gfpf_z : z \in \ISp_\infty\right\}.\]
\end{corollary}

\begin{proof}
It follows by induction from Theorem~\ref{sp-lenart-thm}
that $(1+\beta x_j)(1+\beta x_k)\Gfpf_v$ is a possibly infinite $\ZZ[\beta]$-linear combination of $\Gfpf_z$'s.
This combination must be finite by Corollary~\ref{basis-cor}, 
since no Grothendieck polynomial $\fkG_w$ appears in the expansion of $\Gfpf_y$ and $\Gfpf_z$ for distinct $y,z \in \ISp_\infty$
by \cite[Theorem 3.12]{MP}.
\end{proof}

A \emph{visible descent} of $z \in \ISp_\infty$ is an integer $i$ such that $z(i+1) < \min\{i,z(i)\}$.
The following corollary is a symplectic analogue of Lascoux's transition equations for Grothendieck polynomials in \cite{Lascoux2000}.

\begin{corollary}\label{sp-lenart-cor}
Let $k \in \PP$ be the last visible descent of $z \in \ISp_\infty$.
Define $l $ to be the largest integer with $k<l$ and $z(l) < \min\{k,z(k)\}$,
and set 
\[v = (k,l)z(k,l)\qquand j = v(k).\] 
Let $1\leq  i_1 < i_2 < \dots < i_p<j$ be the integers with $v \lessF (i,j)v(i,j)$. Then
\[ \beta \Gfpf_z = 
(1+\beta x_j)(1+\beta x_k) \Bigl[ \Gfpf_v\cdot  (1+\beta\u_{i_1j})(1+\beta\u_{i_2j})\cdots (1+\beta\u_{i_pj})\Bigr]-\Gfpf_v.
\]
\end{corollary}

Note that one could rewrite the right side 
without using any minus signs.

\begin{proof}
It suffices by Theorem~\ref{sp-lenart-thm}
to show that $\Asc^+(v,j,k) = \{l\}$.
This is precisely \cite[Lemma 5.2]{HMP5}, but also follows as a self-contained exercise.
\end{proof}


\section{Stable Grothendieck polynomials}\label{stable-sect}

The limit of a sequence of polynomials or formal power series is defined to converge if the coefficient sequence for any fixed monomial is eventually constant.

Given $n \in \NN$ and $w \in S_\infty$,
write $1^n \times w \in S_\infty$ for the permutation 
that maps $i \mapsto i$ for $i \leq n$ and $i + n \mapsto w(i) + n$ for $i \in \PP$.
The \emph{stable Grothendieck polynomial} of $w\in S_\infty$ is defined as the limit
\be G_w := \lim_{n \to \infty} \fkG_{1^n \times w}.\ee
Remarkably, this always converges to a well-defined symmetric function \cite[\S2]{Buch2002}.
Given $n \in \NN$ and $z \in \ISp_\infty$, 
we similarly write $(21)^n \times z \in \ISp_\infty$ for the involution 
mapping $i \mapsto i - (-1)^i$ for $i \leq 2n$ and $i + 2n \mapsto z(i) + 2n$ for $i \in \PP$.
Following \cite{MP}, the \emph{symplectic stable Grothendieck polynomial} of $z \in \ISp_\infty$ is defined  as
\be\label{gsp-eq}
\GSp_z := \lim_{n \to \infty} \Gfpf_{(21)^n \times z}.
\ee
The next lemma is a consequence of \cite[Theorem 3.12 and Corollary 4.7]{MP}:

\begin{lemma}[\cite{MP}] 
\label{stab-lem}
The limit \eqref{gsp-eq} converges for all $z \in \ISp_\infty$.
Moreover, the resulting power series $\GSp_z$ is
 the image of $\Gfpf_z$ under the linear map 
$\ZZ[\beta][x_1,x_2,\dots] \to \ZZ[\beta][[x_1,x_2,\dots]]$ with $\fkG_w \mapsto G_w$
for $w \in S_\infty$.
\end{lemma}

It follows that
$\GSp_z$ is also a symmetric function.
These power series have some stronger symmetry properties, which we explore in this section.

\subsection{$K$-theoretic Schur functions}
\label{subsec:G-lambda}

Besides permutations and involutions, there is also a notion of \emph{stable Grothendieck polynomials}
for partitions, though these would more naturally be called \emph{$K$-theoretic Schur functions}.
 The precise definition is as follows.
 
If $\lambda = (\lambda_1 \geq \lambda_2 \geq \dots \geq \lambda_k > 0)$ is an integer partition, then a \emph{set-valued tableau} of shape $\lambda$
is a map $T : (i,j) \mapsto T_{ij}$ from the Young diagram
\[\YY_\lambda := \{(i,j) \in \PP\times \PP : j \leq \lambda_i\}\]
to the set of finite, nonempty subsets of $\PP$.
For such a map $T$, define
\[x^T := \prod_{(i,j) \in \YY_\lambda} \prod_{k \in T_{ij}} x_k
\qquand |T| := \sum_{(i,j) \in \YY_\lambda} |T_{ij}|.\]
A set-valued tableau $T$ is \emph{semistandard} if one has $\max(T_{ij}) \leq \min(T_{i,j+1})$
and $\max(T_{ij}) < \min(T_{i+1,j})$ for all relevant $(i,j) \in \YY_\lambda$.
Let $\SetSSYT(\lambda)$ denote the set of semistandard set-valued tableaux of shape $\lambda$.

\begin{definition}[\cite{Buch2002}]
The \emph{stable Grothendieck polynomial} of a partition $\lambda$ is 
\[ G_\lambda := \sum_{T \in \SetSSYT(\lambda)} \beta^{|T|-|\lambda|} x^T\in \ZZ[\beta][[x_1,x_2,\dots]].\]
\end{definition}

This definition sometimes appears in the literature with the parameter $\beta$
set to $\pm 1$,
but if we write $G^{(\beta)}_\lambda = G_\lambda$ then 
$(-\beta)^{|\lambda|}G^{(\beta)}_\lambda = G^{(-1)}_\lambda(-\beta x_1,-\beta x_2,\dots)$.
Setting $\beta=0$ transforms $G_\lambda$ to the usual Schur function $s_\lambda$.
For example, if $\lambda = (1)$ then
$
G_{(1)} 
=
s_{(1)} + \beta s_{(1,1)} + \beta^2 s_{(1,1,1)} + \dots.$

The functions $G_\lambda$ are related to $G_w$ for $w \in S_\infty$ by the following result of Buch \cite{Buch2002}.
For a partition $\lambda$ with $k$ parts,
define $w_\lambda \in S_\infty$ to be the permutation
with $w_\lambda(i) =i+ \lambda_{k+1-i}$ for $i \in [k]$ and $w_\lambda(i) < w_\lambda(i+1)$
for all $i>k$.

\begin{theorem}[{\cite[Theorem 3.1]{Buch2002}}]
\label{buch-thm1}
If $\lambda$ is any  partition then $G_{w_\lambda} = G_\lambda$.
\end{theorem}

Write $\sP$ for the set of all partitions.

\begin{theorem}[{\cite[Theorem 1]{BKSTY}}]
\label{buch-et-al-thm}
If $w \in S_\infty$ then 
$G_w \in \NN[\beta]\spanning\left\{ G_\lambda : \lambda \in \sP\right\}.$
\end{theorem}

A symplectic analogue of Theorem~\ref{buch-et-al-thm} is known \cite[Theorem 1.9]{Mar}.
Our goal in the rest of this section is to prove a symplectic analogue of Theorem~\ref{buch-thm1}.

\begin{remark*}
 Theorem~\ref{buch-et-al-thm} is a corollary of a stronger result \cite[Theorem 1]{BKSTY}, which 
 gives a formula for the expansion of $G_w$ into $G_\lambda$'s in terms of \emph{increasing tableaux}. 
 Knowing this formula, one can recover 
Theorem~\ref{buch-thm1} by checking that there is a 
unique increasing tableau whose reading word is a \emph{Hecke word} for a Grassmannian permutation.
It may be possible to use a similar strategy to prove our symplectic analogue of Theorem~\ref{buch-thm1}
(given as Theorem~\ref{f-grass-thm}) from \cite[Theorem 1.9]{Mar}. We present a different algebraic proof here, which is independent of \cite{Mar}.
 \end{remark*}

\subsection{Stabilization}
\label{stable-limit-sect}

We refer to the linear map $\ZZ[\beta][x_1,x_2,\dots] \to \ZZ[\beta][[x_1,x_2,\dots]]$ with $\fkG_w \mapsto G_w $
as \emph{stabilization}.
It will be useful in the next two sections to have a description of this operation 
in terms of divided differences.

As in Section~\ref{dd-sect}, let $\cL = \ZZ[\beta][x_1^{\pm1},x_2^{\pm1},\dots]$.
For $i \in \PP$,
write  $\vartheta_i$ for the \emph{isobaric divided difference operator} defined by the formula
\be\label{vartheta-eq}
 \vartheta_i f = \varpi_i (x _i f) 
=  f + x_{i+1}(1+\beta x_i) \partial_i f
\qquad\text{for  $f \in \cL $.}
\ee
We have $\vartheta_i f=f$ if and only if $s_if =f$, 
in which case $\vartheta_i(fg) = f \cdot \vartheta_i g$.
These  operators are idempotent with
$
\vartheta_i\vartheta_i = \vartheta_i
$
for all $i \in \PP$, and we have
\be
\vartheta_i \vartheta_j = \vartheta_j \vartheta_i
\qquand\vartheta_i\vartheta_{i+1} \vartheta_i = \vartheta_{i+1}\vartheta_i \vartheta_{i+1}
\ee
for all $i,j \in \PP$ with $|i-j| > 1$. 
For $w \in S_\infty$ we can therefore define
\[\vartheta_w = \vartheta_{i_1}\vartheta_{i_2}\cdots \vartheta_{i_l}\]
where $w = s_{i_1}s_{i_2} \cdots s_{i_l}$ is any reduced expression.

Given $f \in \ZZ[\beta][[x_1,x_2,\dots]]$ and $n \in \NN$,
write $f(x_1,\dots,x_n)$ for the polynomial obtained by setting 
$x_{n+1}=x_{n+2}=\dots=0$
and let $w_n = n\cdots 321 \in S_n$.

\begin{proposition}\label{stab-prop}
If $v\in S_n$ then 
$\fkG_{1^N\times v}(x_1,\dots,x_n) = \vartheta_{w_n} \fkG_v$ for all $N\geq n$.
\end{proposition}

\begin{proof}
Fix $v \in S_n$ and define $\tau_N := \vartheta_1 \vartheta_2\cdots \vartheta_{N-1}$.
We then have 
\[\tau_{n+1} \fkG_v = \varpi_1 \varpi_2 \cdots \varpi_n (x_1x_2\cdots x_n \fkG_v).\]
Let $u = (v_1+1)(v_2+2)\cdots(v_n+1)1 \in S_{n+1}$. Since 
\[
x_1x_2\cdots x_n \fkG_v = x_1x_2\cdots x_n \varpi_{v^{-1} w_n} \fkG_{w_n}
= \varpi_{v^{-1} w_n} \fkG_{w_{n+1}} = \fkG_u,
\]
it follows that  $\tau_{n+1} \fkG_v = \fkG_{1^1\times v}$ and
$\fkG_{1^N\times v} = \tau_{n+N}\cdots \tau_{n+2}\tau_{n+1}\fkG_v$ for all $N \in \PP$.
Define $r_n(f) := f(x_1,x_2,\dots,x_n)$.
Then  \eqref{vartheta-eq} implies that
\be r_n (\vartheta_i f) = \begin{cases}
r_n (f) &\text{if }n \leq i \\
\vartheta_i r_n( f ) &\text{if }i<n
\end{cases}
\ee
so $r_n( \tau_N f) = \tau_n r_N( f)$ for $n \leq N$
and $r_n(\tau_N f)  = \tau_N r_n( f)$ for $N < n$.
Since $r_n (\fkG_v) = \fkG_v$ and 
$(\tau_n)^n = \vartheta_{w_n}$,
we have $r_n(\fkG_{1^{N}\times v})  
= \vartheta_{w_n} \fkG_v$
for $N\geq n$.
\end{proof}

\begin{corollary}\label{stab-cor}
If $v \in S_\infty$ and $z \in \ISp_\infty$ then 
\[ 
G_v = \lim_{N \to \infty} \vartheta_{w_N} \fkG_v
\qquand
\GSp_z = \lim_{N\to \infty} \vartheta_{w_N} \Gfpf_z.
\]
\end{corollary}

\begin{proof}
These identities are clear from Lemma~\ref{stab-lem} and Proposition~\ref{stab-prop}.
\end{proof}

For any polynomials $x$ and $y$, let 
\be\label{oplus-def}
x\oplus y := x + y + \beta xy \qquand x\ominus y := \tfrac{x-y}{1+\beta y}
\ee
For integers $0<a \leq b$, define 
\[
 \partial_{b\searrow a} := \partial_{b-1}\partial_{b-2}\cdots \partial_a
\qquand
\varpi_{b\searrow a} := \varpi_{b-1}\varpi_{b-2}\cdots \varpi_a
\]
 so that $\partial_{a\searrow a} = \varpi_{a\searrow a} = 1$.
Finally, let $\Delta^{(\beta)}_{m,n}(x) := \prod_{j=2}^n (1+\beta x_{m+j})^{j- 1}$.

\begin{lemma}\label{tech-lem}
If $m \in \NN$ and $n \in \PP$ then 
\[\varpi_{1^m\times w_n} f 
=
\partial_{1^m\times w_n} ( \Delta^{(\beta)}_{m,n}(x) f )
= \sum_{w \in S_n} w\( \frac{f}{\prod_{1 \leq i < j \leq n} x_{m+i} \ominus x_{m+j}}\)
\]
where in the last sum $S_n$ acts by permuting the variables $x_{m+1},x_{m+2},\dots,x_{m+n}$.
\end{lemma}

\begin{proof}
The second equality is \cite[Proposition 2.3.2]{Manivel}.
The first equality follows by induction: the base case when $n=1$ holds by definition, and if $n>1$ then 
 $\varpi_{1^m\times w_n} = \varpi_{(m+n) \searrow (m+1)}\varpi_{1^{m+1}\times w_{n-1}}$
and the desired identity is easy to deduce using the fact that
$\varpi_{b\searrow a} f = \partial_{b\searrow a}((1+\beta x_{b})\cdots (1+\beta x_{a+2})(1+\beta x_{a+1}) f)$.
\end{proof}

For any integer sequence $\lambda=(\lambda_1,\lambda_2,\dots)$ with finitely many
nonzero terms, define $x^\lambda := x_1^{\lambda_1}x_2^{\lambda_2}\cdots$.
Let $\delta_n :=(n-1,n-2,\dots,2,1,0)$
for $n \in \PP$. 

\begin{lemma}\label{vartheta-lem}
If $n \in \PP$ then $\vartheta_{w_n} f = \varpi_{w_n}(x^{\delta_n} f)$ for all $f \in \cL$.
\end{lemma}

\begin{proof}
The expression $w_n =(s_1)(s_2s_1)(s_3s_2s_1)\cdots(s_{n-1} \cdots s_3s_2s_1)$
is reduced
and one can check, noting that $\varpi_i(x_1x_2\cdots x_n f) = x_1x_2\cdots x_n \cdot\varpi_i f$ for $i<n$, that
$\varpi_{n-1}\cdots  \varpi_2 \varpi_1(x^{\delta_n} f) = x^{\delta_{n-1}}  \vartheta_{n-1}\cdots \vartheta_2 \vartheta_1 f
$. The lemma follows by induction from these identities.
\end{proof}


\begin{corollary} \label{g-prop}
If $\lambda$ is a partition then
$
G_\lambda = 
\lim_{n\to \infty} \vartheta_{w_n}( 
 x^{\lambda}  )
.$
\end{corollary}

\begin{proof}
Apply Lemmas~\ref{tech-lem} and \ref{vartheta-lem} to \cite[Eq.\ (2.14)]{IkedaNaruse}, for example.
\end{proof}

\subsection{$K$-theoretic Schur $P$-functions}

The natural symplectic analogues of Theorems~\ref{buch-thm1}
and \ref{buch-et-al-thm} involve shifted versions of the symmetric functions $G_\lambda$,
which we review here.

Define the \emph{marked alphabet} to be the totally ordered set 
of primed and unprimed integers $\MM := \{1'<1<2'<2<\dots\}$,
and write
$|i'|  := |i| = i$ for $i \in \PP$.
If $\lambda =(\lambda_1>\lambda_2>\dots>\lambda_k>0)$ is a strict partition, then a \emph{shifted set-valued tableau} of shape $\lambda$
is a map $T : (i,j) \mapsto T_{ij}$ from the shifted diagram
\[\SS_\lambda := \{(i,i+j-1) \in \PP\times \PP : 1\leq j \leq \lambda_i\}\]
to the set of finite, nonempty subsets of  $\MM $.
Given such a map $T$, define
\[x^T := \prod_{(i,j) \in \SS_\lambda} \prod_{k \in T_{ij}} x_{|k|}
\qquand
 |T| := \sum_{(i,j) \in \SS_\lambda} |T_{ij}|.\]
A shifted set-valued tableau $T$ is \emph{semistandard} if 
for all relevant $(i,j) \in \SS_\lambda$:
\begin{itemize}
\item[(a)] $\max(T_{ij}) \leq \min(T_{i,j+1})$ and $T_{ij} \cap T_{i,j+1} \subset \{1,2,3,\dots\}$.
\item[(b)] $\max(T_{ij}) \leq \min(T_{i+1,j})$ and $T_{ij} \cap T_{i+1,j} \subset \{1',2',3',\dots\}$.
\end{itemize}
In such tableaux, an unprimed number can appear at most once in a column, while a primed number 
can appear at most one in a row.
Let $\SetSSMT(\lambda)$ denote the set of semistandard shifted set-valued tableaux of shape $\lambda$.

\begin{definition}[\cite{IkedaNaruse}]
The \emph{$K$-theoretic Schur $P$-function}
of a strict partition $\lambda$ is the power series
$
\GP_\lambda := \sum_{
T
} \beta^{|T|-|\lambda|} x^T
$
where the summation is over tableaux  $T \in \SetSSMT(\lambda)$ with no primed numbers in any position on the main diagonal.
\end{definition}

This definition is due to Ikeda and Naruse \cite{IkedaNaruse}, who
also show that each
$\GP_\lambda$ is symmetric in the $x_i$ variables
\cite[Theorem 9.1]{IkedaNaruse}.
Setting $\beta=0$ transforms $\GP_\lambda$  
to the classical \emph{Schur $P$-function} $P_\lambda$.

 

\begin{proposition} \label{gp-prop}
If $\lambda$ is a strict partition with $r$ parts then 
\[ \ba
\GP_\lambda& = 
\lim_{n\to \infty} \vartheta_{w_n}\( 
 x^{\lambda} \prod_{i=1}^r \prod_{j=i+1}^n\frac{x_i\oplus x_j}{x_i} \)
 \ea
 \]
 where we set $x\oplus y := x + y + \beta xy$ as in \eqref{oplus-def}.
\end{proposition}

\begin{proof}
As in \eqref{oplus-def}, set $x\ominus y := \frac{x-y}{1+\beta y}$.
Fix a strict partition $\lambda$ with $r$ parts.
Ikeda and Naruse's first definition of $\GP_\lambda$ (see \cite[Definition 2.1]{IkedaNaruse}) is
\be 
\GP_\lambda = 
\lim_{n\to \infty} \frac{1}{(n-r)!} \sum_{w \in S_n} w\( 
x^\lambda \prod_{i=1}^r \prod_{j=i+1}^n \frac{x_i\oplus x_j}{x_i\ominus x_j}
\).
\ee
We can rewrite this  as
\be\label{rewrite-gp}
\GP_\lambda = 
\lim_{n\to \infty} \sum_{w \in S_n/S_{n-r}} w\( 
x^\lambda \prod_{i=1}^r \prod_{j=i+1}^n \frac{x_i\oplus x_j}{x_i\ominus x_j}
\)
\ee
where $S_{n-r}$ acts on the variables $x_{r+1},x_{r+2},\dots,x_n$.
Lemma~\ref{tech-lem} implies that 
$ 1 =
\sum_{w \in S_{n-r}} w\( \frac{\prod_{i=r+1}^n x_{i}^{n-i}}{\prod_{r+1 \leq i <j \leq n} x_{i} \ominus x_{j}}\)
$
since the left side is $\varpi_{1^r \times w_{n-r}} \fkG_{1^r\times w_{n-r}} = \fkG_{1}$.
Multiplying the right side of \eqref{rewrite-gp} by this expression gives
\[ 
\GP_\lambda = 
\lim_{n\to \infty} \sum_{w \in S_n} w\( 
\frac{x^{\delta_n}}{\prod_{1\leq i < j \leq n} x_i \ominus x_j} \cdot x^\lambda \prod_{i=1}^r \prod_{j=i+1}^n \frac{x_i\oplus x_j}{x_i}  \)
\]
which is equivalent to
the desired formula
 by Lemmas~\ref{tech-lem} and \ref{vartheta-lem}.
\end{proof}

\subsection{Grassmannian formulas}

 We are ready to state the main new results of this section.
Fix $z \in \ISp_\infty$.
The \emph{symplectic code} of $z$ is the sequence of integers
\[
\csp(z) = (c_1,c_2,\dots),
\quad\text{where $c_i:= |\{ j \in \PP : z(i) > z(j) < i < j\}|$}.\]
The \emph{symplectic shape} $\shSp(z)$ of $z $ is the transpose of the partition 
sorting $\csp(z)$. 
For example,
if $n \in 2\PP$ and $z = n\cdots 321 \in \ISp_\infty$
then
\[\ba
\DSp(z) &= \{ (i,j)  \in \PP\times \PP : j < i \leq n-j\},\\
\csp(z) &=(0,1,2,\dots, \tfrac{n}{2}-1, \tfrac{n}{2}-1,  \dots, 2, 1, 0,0,\dots),\\
\shSp(z) &= (n-2,n-4,n-6,\dots,2).
\ea\] 
Define $y \in I_\infty $
to be the involution with 
\[ y(i) = \begin{cases}
i&\text{if $z(e)>\max\{i,z(i)\}$ for all  $\min\{i,z(i)\} < e < \max\{i,z(i)\}$} \\
z(i) &\text{otherwise}.
\end{cases}
\]
This means that $y(i) = i$ if $z(i) = i \pm 1$.
In the sequel, we set $\dearc(z)= y.$

The operation $\dearc$ is easy to understand in terms of the arc diagram 
$\{ \{i,z(i)\}: i \in \PP\}$ of $z \in \ISp_\infty$. The arc diagram of $\dearc(z)$ is formed from that of $z$ by deleting each edge $\{i<j\}$
with $e<z(e)$ for all $i<e<j$.

Recall that $i$ is a visible descent of $z \in \ISp_\infty$ if $z(i+1) < \min\{i,z(i)\}$

\begin{definition}[{\cite[\S4]{HMP5}}]
An element $z\in \ISp_\infty$ is \emph{FPF-Grassmannian} if 
\[\dearc(z) = (\phi_1,n+1)(\phi_2,n+2)\cdots(\phi_r,n+r)\]
for a sequence of integers $1\leq \phi_1 < \phi_2 < \dots < \phi_r \leq n$.
In this case, one has \[\shSp(z) = (n-\phi_1,n-\phi_2,\dots,n-\phi_r)\] 
by \cite[Lemma 4.16]{HMP5}, and $n$ is the last visible descent
of $z$. 
\end{definition}

We allow $r=0$ in this definition; this corresponds to the FPF-Grassmannian involution $\wfpf \in \ISp_\infty$ with $\dearc(\wfpf)=1$.
For a given strict partition $\lambda$ with $r <n$ parts,
there is exactly one FPF-Grassmannian involution $z \in \ISp_\infty$ with shape $\shSp(z) =\lambda$
and last visible descent $n$.

\begin{example}
The involution $z = 47816523 = (1,4)(2,7)(3,8)(5,6)  \in \ISp_\infty$
is FPF-Grassmannian with $\dearc(z) = (2,7)(3,8)$ and $\shSp(z) = (4,3)$.
\end{example}

Define
$\vartheta_{b\searrow a} := \vartheta_{b-1}\vartheta_{b-2}\cdots \vartheta_a$
for  $0< a \leq b$,
with $\vartheta_i$ given by \eqref{vartheta-eq}.

\begin{proposition}\label{searrow-prop}
Suppose $z \in \ISp_\infty-\{\wfpf\}$ is FPF-Grassmannian with last visible descent $n$
and shape $\shSp(z) = (n-\phi_1,n-\phi_2,\dots,n-\phi_r)$,
so that 
\[\dearc(z) = (\phi_1,n+1)(\phi_2,n+2)\cdots(\phi_r,n+r)\]
for some integers $1\leq \phi_1 <\phi_2  < \dots < \phi_r \leq n$.
Then 
\[ \Gfpf_z = \vartheta_{\phi_1\searrow 1}\vartheta_{\phi_2\searrow 2}\cdots \vartheta_{\phi_r\searrow r}\(  x^{\shSp(z)} \prod_{i=1}^r \prod_{j=i+1}^n \frac{x_i\oplus x_j}{x_i} \)
\]
where $x_i \oplus x_j := x_i + x_j + \beta x_i x_j$.
\end{proposition}

We need two lemmas to prove this proposition.

\begin{lemma}\label{eee-lem}
If $a\leq b$ then $\varpi_{b\searrow a}(x_a^e) = (-\beta)^{b-a-e}$ for $e \in \{0,1,2,\dots,b-a\}$.
\end{lemma}

\begin{proof}
Since $\varpi_i(1) = -\beta$,
it is enough to check that $\varpi_{b\searrow a}(x_a^{b-a})=1$.
As
\[\varpi_{a}(x_a^{b-a}) =-\beta x_a^{b-a} + (1+\beta x_a) \partial_a (x_a^{b-a})\]
we have 
$
\varpi_{b\searrow a}(x_a^{b-a}) 
= (-\beta x_a)^{b-a} + (1+\beta x_a) \varpi_{b\searrow (a+1)}(\partial_a x_a^{b-a}).
$
By induction
\[\varpi_{b\searrow (a+1)}(\partial_a x_a^{b-a})
=
\varpi_{b\searrow (a+1)}\(\sum_{i=0}^{b-a-1} x_a^i x_{a+1}^{b-a-1-i}\)
=\sum_{i=0}^{b-a-1} (-\beta x_a)^i 
\]
so the lemma follows.
\end{proof}

 \begin{lemma}\label{yyy-lem}
 If $a\leq b$ and $s_i f = f$ for $a<i<b$, then 
 $\vartheta_{b\searrow a} (f) = \varpi_{b\searrow a}(x_a^{b-a} f) $.
 \end{lemma}
 
 \begin{proof}
 Assume $a<b$. It holds by induction 
 that
 \[
\vartheta_{b\searrow a} (f) = 
\vartheta_{b\searrow(a+1)} (\vartheta_a f) = 
\varpi_{b\searrow(a+1)} \(x_{a+1}^{b-a-1}\vartheta_a f\)
.\]
Since
$
\varpi_a(x_a^{b-a} f) = x_{a+1}^{b-a-1}\(\vartheta_a f + \beta x_a f\) + x_af\cdot \varpi_a(x_a^{b-a-1}) 
$,
we have
\[
\vartheta_{b\searrow a}(f) = \varpi_{b\searrow a}(x_a^{b-a}f)
-
x_a f \(\beta \cdot \varpi_{b\searrow (a+1)}( x_{a+1}^{b-a-1}) 
 +
\varpi_{b\searrow a}(x_a^{b-a-1})\).
\]
From here, it suffices to show that 
$
\beta \cdot \varpi_{b\searrow (a+1)}( x_{a+1}^{b-a-1}) 
 +
\varpi_{b\searrow a}(x_a^{b-a-1}) = 0
$
and this is immediate from Lemma~\ref{eee-lem}.
 \end{proof}

\begin{proof}[Proof of Proposition~\ref{searrow-prop}]
Setting $\beta=0$ recovers \cite[Lemma 4.18]{HMP5}; the proof for generic $\beta$ 
is similar.
Let   $\Psi_{n,r}(x) = \prod_{i=1}^r \prod_{j=i+1}^n \frac{x_i\oplus x_j}{x_i}$.
Then $x^{\shSp(z)}\Psi_{n,r}(x)$ is symmetric in  $x_{r+1},x_{r+2},\dots,x_n$.
For any $j\in[r]$, the expression
\[\theta_j := \vartheta_{\phi_j\searrow j}\vartheta_{\phi_{j+1}\searrow (j+1)}\cdots \vartheta_{\phi_r\searrow r}\(  x^{\shSp(z)}
\Psi_{n,r}(x)\)\]
is symmetric in  $x_{j},x_{j+1},\dots,x_{\phi_j}$
since if $i \in \{j,j+1,\dots,\phi_j-1\}$ 
then  either $i = \phi_j-1$ and 
$\vartheta_i \theta_j = \theta_j$ 
or
$i < \phi_j-1$ and 
\[
\vartheta_i \theta_j =\vartheta_i \vartheta_{\phi_j\searrow j}\theta_{j+1} =
\vartheta_{\phi_j\searrow j} \vartheta_{i+1} \theta_{j+1} = 
\vartheta_{\phi_j\searrow j} \theta_{j+1} = \theta_j
\]
by the braid relations for $\vartheta_i$ and induction.
Using Theorem~\ref{dom-thm},
we can rewrite  
\[
\ba
x^{\shSp(z)} \Psi_{n,r}(x) &=
 x_1^{1-\phi_1} x_2^{2-\phi_2}\cdots x_r^{r-\phi_r} \prod_{i=1}^r \prod_{j=i+1}^n x_i \oplus x_j 
\\&=x_1^{1-\phi_1} x_2^{2-\phi_2}\cdots x_r^{r-\phi_r} \Gfpf_{w}
\ea
\]
where $w \in \ISp_\infty$ is the $\Sp$-dominant involution satisfying $\dearc(w)=(1,n+1)(2,n+2)\cdots(r,n+r)$.
Hence by
Lemma~\ref{yyy-lem} we have
\[
 \vartheta_{\phi_1\searrow 1}\vartheta_{\phi_2\searrow 2}\cdots \vartheta_{\phi_r\searrow r}\(  x^{\shSp(z)} \Psi_{n,r}(x) \)
 =
  \varpi_{\phi_1\searrow 1}\varpi_{\phi_2\searrow 2}\cdots \varpi_{\phi_r\searrow r}\( \Gfpf_w \).
 \]
 It is straightforward from Theorem-Definition~\ref{sp-thm1} to show that this
 is  $\Gfpf_z$.
\end{proof}

We can now prove the obvious identity suggested by the notation ``$\GSp_z$'':
 
 \begin{theorem}\label{f-grass-thm}
If $z \in \ISp_\infty$ is FPF-Grassmannian then $\GSp_z = \GP_{\shSp(z)}$.
\end{theorem}

\begin{proof}
Assume $z \in \ISp_\infty$ is as in Proposition~\ref{searrow-prop}. Then 
\[ \vartheta_{w_N} \Gfpf_z=
\vartheta_{w_N} \vartheta_{\phi_1\searrow 1}\cdots \vartheta_{\phi_r\searrow r} (x^{\shSp(z)}\Psi_{n,r}(x))
=
 \vartheta_{w_N}  (x^{\shSp(z)}\Psi_{n,r}(x))
\]
so 
$\ds\GSp_z = \lim_{N\to \infty} \vartheta_{w_N} \Gfpf_z = \GP_{\shSp(z)}$
by Corollary~\ref{stab-cor} and Proposition~\ref{gp-prop}.
\end{proof}

Let $\SLambda$ denote the set of strict partitions.

\begin{corollary} \label{cor:GP-into-G}
If $\lambda \in \SLambda$ then 
$\GP_\lambda \in \NN[\beta]\spanning\left\{  G_\mu : \mu \in \sP\right\}.$
\end{corollary}

\begin{proof}
If $\lambda \in \SLambda$ then there is an FPF-Grassmannian
 $z \in \ISp_\infty$ with $\shSp(z)=\lambda$,
 and \cite[Corollary 4.7]{MP} shows that $\GSp_z \in \NN[\beta]\spanning\{ G_w : w \in S_\infty\}$.
 The corollary therefore follows from
Theorems~\ref{buch-et-al-thm} and \ref{f-grass-thm}.
\end{proof}

\begin{remark}
As noted in the introduction, one can also derive this corollary from \cite{HKPWZZ,PylPat}.
One needs to compare
\cite[Theorems 1.4 and 2.2 and Proposition 3.5]{HKPWZZ} with
\cite[Lemma 3.2 and Theorems 3.16, 6.11, and 6.24]{PylPat}.
\end{remark}

There is a ``stable'' version of the transition equation for $\Gfpf_z$. 
Let $S_\ZZ$ denote the group of permutations of $\ZZ$ with finite support.
Write $\Theta_\ZZ$ for the permutation of $\ZZ$ with
$i \mapsto i -(-1)^i$ and let \[\ISp_\ZZ = \{ w\cdot \Theta_\ZZ \cdot w^{-1} : w \in S_\ZZ\}.\]
Define $\ellfpf(z)$ for $z \in \ISp_\ZZ$ by modifying the formula \eqref{ellfpf-eq}
to count pairs $(i,j) \in \ZZ\times \ZZ$;
then $\ellfpf(\Theta_\ZZ) = 0$ and
\eqref{ellfpf-eq2} still holds.
We again write $y \lessF z$ for $y,z \in \ISp_\ZZ$ if $\ellfpf(z) = \ellfpf(y)+1$
and $z = tyt$ for a transposition $t \in S_\ZZ$. 

Identify $\ISp_\infty$ with the subset of  $z \in \ISp_\ZZ$
with $z(i) = \Theta_\ZZ(i)$ for all $i \leq 0$. 
Let $\sigma : \ZZ \to \ZZ$ be the map $ i \mapsto i + 2$.
Conjugation by $\sigma$ preserves $ \ISp_\ZZ$,
and every $z \in \ISp_\ZZ$
has $\sigma^n  z \sigma^{-n} \in \ISp_\infty$ for all sufficiently large $n \in \NN$.
We  define  
\[ \GSp_z := \lim_{n\to \infty} \GSp_{\sigma^n  z \sigma^{-n}}
\qquad\text{for }z \in \ISp_\ZZ.\]
Also let $\GSp_z \fku_{ij} := \GSp_{(i,j)z(i,j)}$ for $i<j$ and extend by linearity.
In this context, $\fku_{ij}$ is a formal symbolic operator, not a well-defined linear map.

\begin{corollary}\label{stab-version-cor}
Fix $v \in \ISp_\ZZ$ and $j,k \in \ZZ$ with $v(k) = j<k = v(j)$.
Suppose  
\[  i_1  < i_2<\dots < i_p <j< k
 < l_q  < \dots < l_2 < l_1\]
are the integers such that $v \lessF (i,j)v(i,j)$ and $v\lessF (k,l)v(k,l)$.
Then
\[
\GSp_v\cdot  (1+\beta\u_{i_1j})\cdots (1+\beta\u_{i_pj})
=
\GSp_v\cdot (1+\beta\u_{kl_1})\cdots (1+\beta\u_{kl_q}).
\]
\end{corollary}

\begin{proof}
Define $\Asc^-(v,j,k) = \{i_1,i_2,\dots,i_p\}$ and $\Asc^+(v,j,k)= \{l_1,l_2,\dots,l_q\}$.
If $m \in \NN$ is sufficiently large 
then $\Asc^\pm(\Theta^{2m}\times v, 2m+j,2m+k) = 2m + \Asc^\pm(v,j,k)$,
so we obtain this result by taking the limit of Theorem~\ref{sp-lenart-thm}.
\end{proof}

The preceding corollary is a $K$-theoretic generalization of \cite[Theorem 3.6]{HMP5}.
The latter result has an ``orthogonal'' variant given by \cite[Theorem 3.2]{HMP4}.

\begin{corollary}\label{sp-lenart-cor2}
Let $k \in \PP$ be the last visible descent of $z \in \ISp_\ZZ$.
Define $v \in \ISp_\ZZ$ as in Corollary~\ref{sp-lenart-cor}
and
let $I=\{ i_1 < i_2 < \dots < i_p\}$ be the (possibly nonpositive) integers with $i<j :=v(k)$ and $v \lessF (i,j)v(i,j)$.
Then
\[  \GSp_z = 
 \sum_{\varnothing \neq A \subset I } \beta^{|A|-1} \GSp_v \u_{Aj}\]
 where if $A = \{a_1<a_2<\dots <a_q\} \subset I$ then $\u_{Aj} := \u_{a_1j}\u_{a_2j}\cdots \u_{a_qj}$.
\end{corollary}

\begin{proof}
 The proof is the same as for Corollary~\ref{sp-lenart-cor},
 now using Corollary~\ref{stab-version-cor}. 
\end{proof}

This gives a positive recurrence for $\GSp_z$.
We expect that one could use this recurrence
and the inductive strategy in \cite{BilleyTransitions,HMP5,LS1985}
to prove the following
theorem.
However, a direct bijective proof is already available in \cite{Mar}:

\begin{theorem}[{\cite[Theorem 1.9]{Mar}}]
\label{mar-thm}
If $z \in \ISp_\infty$ then 
\[\GSp_z \in \NN[\beta]\spanning\left\{ \GP_\lambda : \lambda \in \SLambda\right\}.\]
\end{theorem}

Combining Theorems~\ref{f-grass-thm} and \ref{mar-thm}
gives this corollary:

\begin{corollary}
If $z \in \ISp_\infty$ then 
\[\GSp_z \in \NN[\beta]\spanning\left\{  \GSp_y : y \in \ISp_\infty\text{ is FPF-Grassmannian}\right\}.\]
\end{corollary}

\end{document}